\documentclass[10pt]{amsart}

\usepackage{enumerate}
\usepackage{graphicx}
\usepackage{float}
\usepackage{placeins}
\usepackage{mdframed}
\usepackage{amssymb}
\usepackage{esint}
\usepackage{cool}
\usepackage[all,cmtip]{xy}
\usepackage{mathtools}
\usepackage{amstext} 
\usepackage{array}   
\usepackage[shortlabels]{enumitem}

\newcolumntype{L}{>{$}l<{$}} 
\newtheorem{theorem}{Theorem}[section]
\newtheorem{lemma}[theorem]{Lemma}
\newtheorem{cor}[theorem]{Corollary}
\newtheorem{prop}[theorem]{Proposition}
\newtheorem{setup}[theorem]{Setup}
\theoremstyle{definition}
\newtheorem{definition}[theorem]{Definition}
\newtheorem{example}[theorem]{Example}

\newtheorem{observation}[theorem]{Observation}

\newtheorem{chunk}[theorem]{}
\theoremstyle{remark}
\newtheorem{remark}[theorem]{Remark}

\newtheorem{the context}[theorem]{The Context}
\newtheorem{question}[theorem]{Question}
\numberwithin{equation}{theorem}
\numberwithin{equation}{section}

\newcommand{\cat}[1]{\mathcal{#1}}

\newcommand{\pd}{\operatorname{pd}}

\newcommand{\rank}{\operatorname{rank}}

\newcommand{\soc}{\operatorname{Soc}}

\newcommand{\type}{\operatorname{type}}

\newcommand{\Span}{\operatorname{Span}}

\newcommand{\tor}{\operatorname{Tor}}
\newcommand{\im}{\operatorname{Im}}

\newcommand{\proj}{\operatorname{Proj}}

\newcommand{\Ker}{\operatorname{Ker}}

\newcommand{\ideal}[1]{\mathfrak{#1}}
\newcommand{\m}{\ideal{m}}

\newcommand{\bbp}{\mathbb{P}}

\renewcommand{\geq}{\geqslant}
\renewcommand{\leq}{\leqslant}
\renewcommand{\ker}{\Ker}
\renewcommand{\hom}{\Hom}

\newcommand{\Hom}{\operatorname{Hom}}

\newcommand{\socle}{\operatorname{Soc}}

\newcommand{\maps}[5]{\xymatrix{#1 \ar[r]^-{#3} & #2 \\
#4 \ar@{|->}[r] & #5 \\}}

\def\mcI{\mathcal I}

\def\w{\wedge}
\def\proj{\operatorname{proj}}
\def\ts{\textstyle}
\def\im{\operatorname{im}}

\newcommand{\p}{\oplus}

\begin{document}
\title[Structure Theory for a Class of Grade $3$ Ideals]{Structure Theory for a Class of Grade 3 Homogeneous Ideals Defining Type 2 Compressed Rings}

\author{Keller VandeBogert }

\date{\today}

\maketitle

\begin{abstract}
    Let $R=k[x,y,z]$ be a standard graded $3$-variable polynomial ring, where $k$ denotes any field. We study grade $3$ homogeneous ideals $I \subseteq R$ defining compressed rings with socle $k(-s) \oplus k(-2s+1)$, where $s \geq3$ is some integer. We prove that all such ideals are obtained by a trimming process introduced by Christensen, Veliche, and Weyman in \cite{christ}. We also construct a general resolution for all such ideals which is minimal in sufficiently generic cases. Using this resolution, we give bounds on the minimal number of generators $\mu(I)$ of $I$ depending only on $s$; moreover, we show these bounds are sharp by constructing ideals attaining the upper and lower bounds for all $s\geq 3$. Finally, we study the Tor-algebra structure of $R/I$. It is shown that these rings have Tor algebra class $G(r)$ for $s \leq r \leq 2s-1$. Furthermore, we produce ideals $I$ for all $s \geq 3$ and all $r$ with $s \leq r \leq 2s-1$ such that $\soc (R/I ) = k(-s) \oplus k(-2s+1)$ and $R/I$ has Tor-algebra class $G(r)$, partially answering a question of realizability posed by Avramov in \cite{av2012}. 
\end{abstract}

\section{Introduction}

Let $(R,\m,k)$ be a regular local ring with maximal ideal $\m$. A result of Buchsbaum and Eisenbud (see \cite{BE}) established that any quotient $R/I$ of $R$ with projective dimension $3$ admits the structure of an associative commutative differental graded (DG) algebra. Later, a complete classification of the multiplicative structure of the Tor algebra $\tor_\bullet^R (R/I , k)$ for such quotients was established by Weyman in \cite{wey89} and Avramov, Kustin, and Miller in \cite{torclass}.

One parametrized family arising from the aforementioned classification of Tor algebras is the class $G(r)$, where $r$ is a parameter arising from the rank of the induced map
$$\delta : \tor_2^R (R/I , k) \to \hom_k ( \tor_1^R(R/I , k) , \tor_3^R (R/I ,k)).$$
If $I \subset R$ is such that $R/I$ is Gorenstein, then it is shown by Avramov and Golod in \cite{av71} that the Koszul homology algebra of $R/I$ is a Poincar\'e duality algebra. Indeed, an equivalent characterization of the Tor algebra class $G(r)$ is that that there exists a subalgebra of the Tor algebra minimally exhibiting Poincar\'e duality, in the sense that there does not exist any nontrivial multiplication outside of this subalgebra (see Definition \ref{classg} for a precise statement). It can be shown that if $R/I$ is Gorenstein (and not a complete intersection) of codimension $3$, then $R/I$ has Tor algebra class $G(\mu(I))$, where $\mu(I)$ denotes the minimal number of generators of $I$. Avramov conjectured in \cite{av2012} that quotients of Tor algebra class $G$ are necessarily Gorenstein rings.

The technique of ``trimming'' a Gorenstein ideal is used by Christensen, Veliche, and Weyman (see \cite{christ}) to produce codimension $3$ non-Gorenstein rings with Tor algebra class $G$. If $(R,\m)$ is a regular local ring and $I = (\phi_1 , \dots , \phi_n) \subseteq R$ is an $\m$-primary ideal with $R/I$ of codimension $3$, then an example of this trimming process is the formation of the ideal $(\phi_1 , \dots , \phi_{n-1} ) + \m \phi_n$. 

The classification of perfect codimension $3$ ideals has seen significant progress recently, starting with the paper \cite{wey2018} (extending the work started in \cite{wey89}), which links this structure theory to the representation theory of Kac-Moody Lie algebras. Resolutions of a given format (sequence of Betti numbers) have an associated graph, and it is conjectured in \cite{christ2} that an ideal is in the linkage class of a complete intersection if and only if this associated graph is a Dynkin diagram. 

In \cite[Question 3.8]{av2012}, Avramov poses a question of realizability; that is, which Tor algebra classes of codimension $3$ local rings can actually occur? Using techniques of linkage, this question is explored in \cite{christ3}, refining the classification provided in \cite{torclass} and showing that every grade $3$ perfect ideal in a regular local ring is in the linkage class of either a complete intersection or an ideal defining a Golod ring.

In this paper, we examine grade $3$ homogeneous ideals $I \subset R:= k[x,y,z]$ (with all variables having degree $1$, and $k$ being a field of arbitrary characteristic) defining an Artinian compressed ring with socle $\soc (R/I) = k(-s) \oplus k(-2s+1)$. The values $s$ and $2s-1$ are interesting because they provide a boundary case for socle degrees; more precisely, it is not possible to have a type $2$ ring with socle $k(-s_1) \oplus k(-s_2)$, where $s_2 \geq 2s_1$. In particular, we prove that all such ideals arise as trimmings of Gorenstein ideals. In Theorem \ref{12.6}, we produce a general resolution for trimmed Gorenstein ideals that is minimal in some generic cases (see Proposition \ref{btab3} for the relevant parameter space and the corresponding open subset). Even in the cases where this resolution is not minimal, there is valuable information to be gained from the relatively simple differentials involved.

We give sharp bounds for the graded Betti numbers for ideals resolved by Theorem \ref{12.6}. Furthermore, we produce a family of ideals attaining all possible intermediate Betti numbers. This family is also used to show that for any integers $r$ and $s$ with $s\geq 3$ and $s \leq r \leq 2s-1$, there exists a grade $3$ ideal $I$ defining an Artinian compressed ring with $\soc (R/I) = k(-s) \oplus k(-2s+1)$ of Tor algebra class $G(r)$ (see Corollary \ref{torach2}), which partially answers the question of realizability mentioned above. More generally, any such $I$ with $s \geq 3$ must have Tor algebra class $G( \mu (I) - 3)$.

The paper is organized as follows: Sections \ref{sec2} and \ref{genbetti} consist of preliminary material and notation. Section \ref{gr3} proves the previously mentioned fact that any grade $3$ ideal $I \subset k[x,y,z]$ defining an Artinian compressed ring with $\soc (R/I ) = k(s) \oplus k(-2s+1)$ is obtained as the trimming of some grade $3$ Gorenstein ideal. Section \ref{theres} builds a resolution of all such ideals, deducing some consequences of the structure of the differentials along the way.

In Section \ref{appex} we explore the initial consequences of the resolution built in Section \ref{theres}. In the standard graded case, we find a remarkably simple criterion to deduce whether the trimmed generating set of a Gorenstein ideal is a minimal generating set (see Proposition \ref{isminl}). In particular, questions about minimal generators are translated into counting degrees of the entries of the presenting matrix of a Gorenstein ideal. 

Section \ref{extbetti} deduces the maximal number of minimal generators of a grade $3$ ideal $I$ defining an Artinian compressed ring with $\soc (R/I ) = k(s) \oplus k(-2s+1)$. Moreover, we produce an ideal achieving this upper bound for every $s \geq 2$, showing that the bound is sharp.

Section \ref{toralgstr} delves into some more nontrivial consequences of the tools developed beforehand. In \cite{christ}, all possible Tor algebra structures of trimmed Gorenstein ideals are enumerated. As a consequence, all possible Tor algebra structures for the ideals of interest may be deduced. Combining this with the bounds on the minimal number of generators, we show that all such ideals are class $G(r)$ for some $s \leq r \leq 2s-1$. Furthermore, using information from the resolution of Section \ref{theres}, we show that every such $r$ value between $s$ and $2s-1$ may be achieved by choosing an ideal from the family introduced in Section \ref{extbetti}.

Finally, Section \ref{evencase} drops the top socle degree by $1$ and gives a rudimentary analysis of  grade $3$ ideals $I \subset k[x,y,z]$ defining compressed rings with $\soc (R/I) = k(-s) \oplus k(-2s+2)$. Such ideals are also trimmed Gorenstein ideals, hence resolved by Theorem \ref{12.6}. The possible Tor algebra structures are much more limited in this case, and we end with a question about the existence of ideals with a prescribed number of minimal generators that are also Tor algebra class $G(r)$, for a specified $r$.

\section{Compressed Rings and Inverse Systems}\label{sec2}

\begin{definition}
Let $A$ be a local Artinian $k$-algebra, where $k$ is a field and $\m$ denotes the maximal ideal. The top socle degree is the maximum $s$ with $\m^s \neq 0$ and the socle polynomial of $A$ is the formal polynomial $\sum_{i=0}^s c_i z^i$, where
$$c_i = \dim_k \frac{\socle (A) \cap \m^i}{\socle (A) \cap \m^{i+1}}.$$
An Artinian $k$-algebra is \emph{standard graded} if it is generated as an algebra in degree $1$. 
\end{definition}

\begin{definition}\label{compdef}
A standard graded Artinian $k$-algebra $A$ with embedding dimension $e$, top socle degree $s$, and socle polynomial $\sum_{i=0}^s c_i z^i$ is \emph{compressed} if
$$\dim_k \m^i/\m^{i+1} = \min \Big\{ \binom{e-1+i}{i} , \sum_{\ell=0}^s c_\ell \binom{e-1+\ell-i}{\ell - i} \Big\}$$
for $i =0, \dots , s$.
\end{definition}

\begin{setup}\label{setup1}
Let $n \geq 1$ be an integer and $k$ denote a field of arbitrary characteristic. Let $V$ be a vector space of dimension $n$ over $k$. Give the symmetric algebra $S(V) =: R$ and divided power algebra $D(V^*)$ the standard grading (that is, $S_1(V) = V$, $D_1 (V^*) = V^*$). The notation $S_i := S_i (V)$ denotes the degree $i$ component of the symmetric algebra on $V$. Similarly, the notation $D_i := D_i (V^*)$ denotes the degree $i$ component of the divided power algebra on $V^*$.

Given a homogeneous $I \subseteq S(V)$ defining an Artinian ring, there is an associated inverse system $0:_{D(V^*)} I$. Similarly, for any finitely generated graded submodule $N \subseteq D(V^*)$ there is a corresponding homogeneous ideal $0:_{S(V)} N$ defining an Artinian ring.  

If $I$ is a homogeneous ideal with associated inverse system minimally generated by elements $\phi_1 , \ \dots , \phi_k$ with $\deg \phi_i = s_i$, then there are induced vector space homomorphisms
$$\Phi_i : S_i \to \bigoplus_{j=1}^k D_{s_j - i}$$
sending $f \mapsto (f \cdot \phi_1 , \dots , f \cdot \phi_k)$. 
\end{setup}

\begin{observation}\label{obs2}
Let $I \subseteq S(V)$ be a homogeneous ideal with associated inverse system minimally generated by elements $\phi_1 , \ \dots , \phi_k$ with $\deg \phi_i = s_i$. If the induced maps $\Phi_i$ of Setup \ref{setup1} have maximal rank for all $i$, then the ring $R/I$ is compressed, as in Definition \ref{compdef}.
\end{observation}

\begin{proof}
By definition, $I_i = \ker \Phi_i$; by the rank-nullity theorem,
\begingroup\allowdisplaybreaks
\begin{align*}
    \dim_k (R/I)_i &= \dim_k \im \Phi_i \\
    &= \min \Big\{ \dim_k S_i , \dim_k \bigoplus_{j=1}^\ell D_{s_j - i } \Big\} \\
\end{align*}
\endgroup
where the latter equality follows by the assumption that $\Phi_i$ has maximal rank.
\end{proof}

\begin{definition}
Let $I \subseteq S(V)$ be a homogeneous ideal with associated inverse system minimally generated by elements $\phi_1 , \ \dots , \phi_k$ with $\deg \phi_i = s_i$. Let $m$ denote the first integer for which $\Phi_m$ is a surjection. Then $m$ is called the \emph{tipping point} of $I$; this is well defined since the rank of the domain and codomain of each $\Phi_i$ is increasing/decreasing in $i$, respectively (and the codomain is eventually $0$).
\end{definition}

\begin{prop}[\cite{miller}, Lemma 1.13]\label{proplol}
Let $\phi$ be a homogeneous element of $D(V^*)$ of degree $s$. Then the tipping point of the ideal $0 :_{S(V)} \phi$ is $\lceil s/2 \rceil$. In addition, the induced maps $\Phi_i$ satisfy the following properties for every integer $i$.
\begin{enumerate}[(a)]
    \item $\hom_{k} (\Phi_i , k) = \Phi_{s-i}$
    \item $\Phi_i$ is surjective if and only if $\Phi_{s-i}$ is injective.
\end{enumerate}
\end{prop}

\begin{definition}
Adopt Setup \ref{setup1} with $R = S(V)$ and let $\psi : V \to R$. The Koszul complex $K_\bullet$ on $\psi$ is the complex obtained by setting 
$$K_i := \bigwedge^i V \otimes R(-i)$$
with differential
$$\delta_i : \bigwedge^i V \otimes R(-i) \to \bigwedge^{i-1} V \otimes R(-i+1)$$
defined as multiplication by $\psi \in V^*$ (where $\bigwedge^\bullet V$ is given the standard module structure over $\bigwedge^\bullet V^*$). 
\end{definition}

The following can be found as Proposition $2.5$ of \cite{boij}:

\begin{prop}\label{caval}
Let $I$ be a homogeneous ideal in $R:= S(V)$ of initial degree $t$, and set $A = R/I$. Then
$$\tor_i^R (A , k)_{i+t-1} \cong \ker (\pi ) \cap \ker (\delta_{i-1}), \quad i=2, \dots , n$$
where $\delta_{i-1}$ is the Koszul differential $\bigwedge^{i-1}V \otimes R_t \to \bigwedge^{i-2}V \otimes R_{t+1}$ and $\pi$ is the quotient map $\bigwedge^{i-1} V \otimes R_t \to \bigwedge^{i-1} V \otimes A_t$.
\end{prop}

\begin{remark}\label{rk1}
Adopt notation and hypotheses of Setup \ref{setup1}. Let $\psi : V \to R $ be such that $\im \psi = R_+$. Observe that $\ker (\pi) \cap \ker ( \delta_{i-1})$ as in Proposition \ref{caval} is precisely $\ker ( \pi ) \cap \im (\delta_i)$ by exactness of the Koszul complex. The latter set may be described as the kernel of the composition of $k$-vector space homomorphisms
$$\xymatrix{\bigwedge^i V \otimes S_t (V) \ar[r]^-{\delta_i} & \bigwedge^{i-1} V \otimes S_{t+1} (V) \\
\ar[r]^-{1 \otimes \Phi_{t+1}} & \bigwedge^{i-1} V \otimes D_{c-t-1} (V^*). \\}$$
Denote the above composition of $k$-vector space homomorphisms by
$$\Theta_i (\phi)  : \bigwedge^i V \otimes S_t (V) \to \bigwedge^{i-1} V \otimes D_{c-t-1} (V^*).$$
\end{remark}

\begin{prop}\label{rank}
Adopt Setup \ref{setup1}. Let $A = R/I$ where $I =  (0 :_R \phi)$, $\phi \in D_c(V^*)$, and let $t$ denote the initial degree of $I$, $n = \pd_R R/I$. Then,
$$\dim_k \tor_i^R (A , k)_{i+t-1} = \binom{t-1+i-1}{i-1} \binom{t-1+n}{n-i} - \rank \Theta_i (\phi)$$
for all $i=2, \dots , n$.
\end{prop}

\begin{proof}
First observe that the minimal homogeneous free resolution of 
$$\im (\delta_i : \bigwedge^{i+1} V \otimes R_{t-1} \to \bigwedge^i V \otimes R_t)$$ 
is obtained by truncating the Koszul complex:
$$0 \to \bigwedge^{i+t} V \otimes R_0 \to \dots \to \bigwedge^{i+1} V \otimes R_{t-1} \to \im(\delta_i) \to 0$$
whence
\begin{equation*}
    \begin{split}
        \dim \im (\delta_i) &= \sum_{j=1}^t (-1)^{j+1} \dim \Big( \bigwedge^{i+j} V \otimes R_{t-j} \Big) \\
        &= \sum_{j=1}^t (-1)^{j+1} \binom{n+1}{i+j} \cdot \binom{n+t-j}{t-j}. \\
    \end{split}
\end{equation*}
By Lemma $1.2$ of \cite{crv}, this sum is equal to $\binom{i+t-1}{i} \cdot \binom{n+t}{i+t}$. Similarly, by construction $\dim \ker (\Theta_i (\phi)) = \dim \big( \ker ( \pi) \cap  \im (\delta_i) \big)$. By exactness of the Koszul complex, $\im (\delta_i) = \ker (\delta_{i-1})$; combining this with Proposition \ref{caval}:
$$\dim \tor_i^R (A , k)_{i+t-1} = \dim \ker ( \Theta_i (\phi) ).$$
By the rank-nullity theorem,
$$\dim \tor_i^R (A , k)_{i+t-1} = \dim \im (\delta_i) - \rank_k \Theta_i (\phi)$$
\end{proof}

\begin{cor}\label{gener}
Adopt notation and hypotheses as in Setup \ref{setup1}. Then there is a nonempty open set $U$ in the Grassmannian parametrizing all $1$-dimensional subspaces of $D_c (V^*)$ such that the Betti numbers of the $k$-algebra $A = R/ (0:_{S(V)} \phi) $ are the same for all $[\phi]$ in $U$ (where $[\phi]$ denotes the class of the subspace spanned by $\phi \in D_c (V^*)$). 
\end{cor}

\begin{proof}
Take the open subset $U$ to be the set of all $1$-dimensional subspaces $[\phi]$ of $D_c (V^*)$ such that $\Theta_i (\phi)$ has maximal rank for each $i=2, \dots , n$. 

We may identify the Grassmannian $\textrm{Gr} (1, D_c (V^*))$ with the projective space $\bbp^{\binom{c+2}{2}-1}$, so it suffices to show that the complement of $U$ is the zero set of homogeneous polynomials in the variables $p_1 , \dots , p_{\binom{c+2}{2}}$, where $[\phi ] = [p_1 : \cdots : p_{\binom{c+2}{2}}]$. 

Let $\epsilon_\beta$ denote any standard basis element of $\bigwedge^i V$, so $\beta = (\beta_1 , \dots , \beta_i)$ with $\beta_1 < \cdots < \beta_i$. Let $m \in S_t (V)$ be any degree $t$ monomial. We compute
$$\Theta_i (\phi)(m) = \sum_{j=1}^i (-1)^{i+1} \epsilon_{\beta \backslash \beta_j} \otimes (\psi(\epsilon_{\beta_j})m)\cdot \phi,$$
implying that the matrix representation of $\Theta_i ( \phi)$ has entries of the form $\pm p_\ell$, for $\ell=1 , \dots , \binom{c+2}{2}$, where the basis chosen for $\bigwedge^{i-1} V \otimes D_{c-t-1} (V^*)$ consists of the tensor products of the standard basis for $\bigwedge^{i-1} V$ and the monomial basis for $D_{c-t-1} (V^*)$. 

The complement of $U$ is the union of the zero sets of the determinant of the above matrix representation for each $i=2, \dots , n$, which is a homogeneous polynomial in the $p_\ell$. As a finite union of closed sets, this set is closed. Thus $U$ is an open set, and by Proposition \ref{rank}, any $[\phi] \in U$ gives rise to an ideal $(0:_R \phi)$ whose Betti numbers are independent of the choice of $\phi$.
\end{proof}

\section{Generic Betti Numbers for Grade $3$ Gorenstein Ideals}\label{genbetti}

\begin{definition}
A standard graded Artinian algebra is \emph{level} if its socle is concentrated entirely in a single degree.
\end{definition}

\begin{prop}[\cite{boij}, Proposition 3.6]\label{equalities}
Let $R/I$ be a standard graded compressed level Artinian algebra of embedding dimension $r$, socle degree $c$, socle dimension $m$, and assume $I$ has initial degree $t$. Then
\begin{equation*}
    \begin{split}
        &\dim_k \tor_i^R (R/I , k)_{t+i-1} - \dim_k \tor_{i-1}^R (R/I , k)_{t+i-1} = \\
        &\binom{t-1+i-1}{i-1} \cdot \binom{t-1+r}{r-i} - m \binom{c-t+r-i}{r-i} \cdot \binom{c-t+r}{i-1} \\
    \end{split}
\end{equation*}
for $i=1, \dots , r-1$.
\end{prop}

\begin{prop}\label{btab1}
Let $R = \textrm{k} [x,y,z]$ be standard graded and $I$ a homogeneous grade $3$ Gorenstein ideal with $R/I$ compressed and $\socle (R/I) = k(-2s+1)$ for some integer $s$. Then $R/I$ has Betti table of the form
\begin{equation*}\begin{tabular}{L|L|L|L|L}
     & 0 & 1 & 2 & 3  \\
     \hline 
   0  & 1 & 0 & 0 & 0 \\
   \hline
   s-1 & 0 &s+1  & b & 0 \\
   \hline 
   s & 0 &  b& s+1 & 0 \\
   \hline 
   2s-1 & 0 & 0 & 0 &1 \\
\end{tabular}
\end{equation*}
where $b$ is some integer. Moreover, $b \leq s$.
\end{prop}

\begin{proof}
Employ Proposition \ref{equalities}, where $r=3$, $c = 2s-1$, $m=1$, and $t = s$ ($= \lceil (2s-1)/2 \rceil$; see Proposition \ref{proplol}). Using the notation
$$T_i := \tor_i^R (R/I,k),$$
we obtain
\begin{equation*}
    \begin{split}
       & \dim (T_1)_{s} = s+1 \\
       &\dim(T_2)_{s+1} - \dim(T_1)_{s+1} = 0 \\
       &\dim (T_2)_{s+2} = s+1. \\
    \end{split}
\end{equation*}
Thus define $b:= \dim_k (T_2)_{s+1}$. The final claim that $b \leq s$ follows from the fact that the Betti table has the following decomposition into standard pure Betti diagrams:
\begin{equation*}
    \begin{split}
      & \frac{b}{2(s+1)^2-2} \cdot \begin{tabular}{L|L|L|L|L}
     & 0 & 1 & 2 & 3  \\
     \hline 
   0  & s+2 & 0 & 0 & 0 \\
   \hline
   s-1 & 0 &2(s+1)^2  & 2(s+1)^2-2 & 0 \\
   \hline 
   s & 0 & 0 & 0 & 0 \\
   \hline 
   2s-1 & 0 & 0 & 0 &s \\
\end{tabular} \\
+ & \frac{(s+1)^2-1 - (s+1)b}{(s+1)^2-1} \cdot  \begin{tabular}{L|L|L|L|L}
     & 0 & 1 & 2 & 3  \\
     \hline 
   0  & 1 & 0 & 0 & 0 \\
   \hline
   s-1 & 0 &s+1  & 0 & 0 \\
   \hline 
   s & 0 & 0 & s+1 & 0 \\
   \hline 
   2s-1 & 0 & 0 & 0 &1 \\
\end{tabular} \\
+ & \frac{b}{2(s+1)^2-2} \cdot  \begin{tabular}{L|L|L|L|L}
     & 0 & 1 & 2 & 3  \\
     \hline 
   0  & s & 0 & 0 & 0 \\
   \hline
   s-1 & 0 &0  & 0 & 0 \\
   \hline 
   s & 0 & 2(s+1)^2-2 & 2(s+1)^2 & 0 \\
   \hline 
   2s-1 & 0 & 0 & 0 &s+2 \\
\end{tabular} \\
    \end{split}
\end{equation*}
If $b \geq s+1$, then the middle coefficient of the above decomposition becomes negative, which is a contradiction to results of Boij-S\"oderberg theory (see, for instance, \cite[Theorem 2]{boijsod}).
\end{proof}

In the following, recall that the notation $[\phi] \in \textrm{Gr} (1 , D_c (V^*))$ means the class of the subspace spanned by the element $\phi \in D_c (V^*)$.

\begin{prop}\label{btab3}
Let $R = S(V)$ be standard graded, where $V$ is a $3$-dimensional vector space over a field $k$. If $s$ is even, then there is a nonempty open set $U$ in the Grassmannian parametrizing all $1$-dimensional subspaces of $D_{2s-1} (V^*)$ such that for all $[\phi] \in U$, $I := (0:_{S(V)} \phi)$ has Betti table
\begin{equation*}\begin{tabular}{L|L|L|L|L}
     & 0 & 1 & 2 & 3  \\
     \hline 
   0  & 1 & 0 & 0 & 0 \\
   \hline
   s-1 & 0 &s+1  & 0 & 0 \\
   \hline 
   s & 0 &  0& s+1 & 0 \\
   \hline 
   2s-1 & 0 & 0 & 0 &1 \\
\end{tabular}
\end{equation*}
If $s$ is odd, then there is a nonempty open set $U$ in the Grassmannian parametrizing all $1$-dimensional subspaces of $D_{2s-1} (V^*)$ such that for all $[\phi] \in U$, $I := (0:_{S(V)} \phi)$ has Betti table
\begin{equation*}\begin{tabular}{L|L|L|L|L}
     & 0 & 1 & 2 & 3  \\
     \hline 
   0  & 1 & 0 & 0 & 0 \\
   \hline
   s-1 & 0 &s+1  & 1 & 0 \\
   \hline 
   s & 0 &  1& s+1 & 0 \\
   \hline 
   2s-1 & 0 & 0 & 0 &1 \\
\end{tabular}
\end{equation*}
\end{prop}

\begin{proof}
The goal is to find minimal values for $b$, where $b$ is as in Proposition \ref{btab1}, since $b$ is minimized precisely when the rank of $\Theta_i (\phi)$ is maximized by Proposition \ref{rank}. To this end, we exhibit an explicit $I$ for each $s$ attaining the Betti table as in the statement and argue that no smaller values of $b$ can be obtained. The matrices used below are those from Proposition $6.2$ of \cite{BE} with minor alterations; in our case, some of the entries are squared. 

Choosing a basis for $V$, we may view $S(V)$ as the standard graded polynomial ring $k[x,y,z]$. Assume first that $s$ is even. Consider the $(s+1)\times (s+1)$ alternating matrix
$$\begin{pmatrix}
0 & x^2 & 0 & \cdots & 0& z^2 \\
-x^2 & 0 & y^2 & \cdots & z^2 & 0 \\
0 & -y^2 & 0 & \cdots & & 0 \\
\vdots & & & & & \vdots \\
-z^2 & 0&\cdots  & & & 0 \\
\end{pmatrix}$$
To see the pattern more clearly, the first two matrices are
$$H^{ev}_1 = \begin{pmatrix}
0 & x^2 & z^2 \\
-x^2 & 0 & y^2 \\
-z^2 & -y^2 & 0 \\
\end{pmatrix}, \quad H^{ev}_2 = \begin{pmatrix}
0 & x^2 & 0 & 0 & z^2 \\
-x^2 & 0 & y^2 & z^2 & 0 \\
0 & -y^2 & 0 & x^2 & 0 \\
0 & -z^2 & -x^2 & 0 & y^2 \\
-z^2 & 0 & 0 & -y^2 & 0 \\
\end{pmatrix}$$
The ideal generated by the $(s) \times (s)$ Pfaffians has grade $3$ according to section $6$ of \cite{BE} (a much more explicit generating set is exhibited in Proposition $7.6$ of \cite{elkustin}), and hence has minimal free resolution
$$0 \to R(-2s+1) \to R(-s-2)^{s+1} \to R(-s)^{s+1} \to R.$$
The above gives an ideal for which $b = 0$, and this is clearly the smallest possible.

Similarly, if $s $ is odd, consider the following $(s+2) \times (s+2)$ matrix:
$$\begin{pmatrix}
0 & x^2 & 0 & \cdots & 0& z \\
-x^2 & 0 & y^2 & \cdots & z^2 & 0 \\
0 & -y^2 & 0 & \cdots & & 0 \\
\vdots & & & & & \vdots \\
-z & 0&\cdots  & & & 0 \\
\end{pmatrix}$$
The first two matrices in this case are
$$H^{odd}_1 = \begin{pmatrix}
0 & x^2 & z \\
-x^2 & 0 & y \\
-z & -y & 0 \\
\end{pmatrix}, \quad H^{odd}_2 = \begin{pmatrix}
0 & x^2 & 0 & 0 & z \\
-x^2 & 0 & y^2 & z^2 & 0 \\
0 & -y^2 & 0 & x^2 & 0 \\
0 & -z^2 & -x^2 & 0 & y \\
-z & 0 & 0 & -y & 0 \\
\end{pmatrix}$$
Again, the ideal generated by the submaximal Pfaffians is grade $3$ Gorenstein with $b= 1$. Moreover, no smaller value of $b$ can be achieved since otherwise the ideal would have an even number of minimal generators, which is impossible by work of Watanabe in \cite{watanabe} or Corollary $2.2$ of \cite{BE}.
\end{proof}

\section{Some Structure in the Grade 3 Setting}\label{gr3}

\begin{setup}\label{setup2}
Let $k$ be a field and let $R = k[x,y,z]$ be a standard graded polynomial ring over a field $k$. Let $I \subset R$ be a grade $3$ homogeneous ideal defining a compressed ring with $\soc (R/I) = k(-s) \oplus k(-2s+1)$, where $s \geq 3$. 

Write $I = I_1 \cap I_2$ for $I_1$, $I_2$ homogeneous grade $3$ Gorenstein ideals defining rings with socle degrees $s$ and $2s-1$, respectively. The notation $R_+$ will denote the irrelevant ideal ($R_{>0}$).
\end{setup}

\begin{prop}\label{alscomp}
Adopt Setup \ref{setup2}. Then the ideal $I_2$ defines a compressed ring.
\end{prop}

\begin{proof}
In view of Observation \ref{obs2} and Proposition \ref{proplol}, it suffices to show that the map $\Phi_i$ associated to $I_2$ is surjective for all $i \geq \lceil s-1/2 \rceil =s$. Assume that the inverse system associated to $I$ is minimally generated by $\phi_1$ and $\phi_2$ of degree $s$ and $2s-1$, respectively. Then $I_2 = 0:_{S(V)} \phi_2$.

By hypothesis, the map $f \mapsto (f \phi_1 , f \phi_2)$ is surjective for all $i \geq s$. Composing with the projection onto the second factor, $f \mapsto f \phi_2$ is also surjective for all $i \geq s$.
\end{proof}

\begin{prop}\label{trimmed}
Adopt Setup \ref{setup2}. There exists a minimal generating set $$(\phi_1 , \dots , \phi_{s+1} , \psi_1 , \dots , \psi_b)$$ for $I_2$ such that
$$I = (\phi_1 , \dots , \phi_s , \psi_1 , \dots , \psi_b) + R_+ \phi_{s+1}$$
where $\deg (\phi_i) = s$, $\deg (\psi_j) = s+1$, and the integer $b$ is obtained as in Proposition \ref{btab1}.
\end{prop}

\begin{proof}
By the definition of a compressed ring,
\begingroup\allowdisplaybreaks
\begin{align*}
    \dim I_s &=\dim R_s - \min \{ \dim_k R_s , \dim_k R_0 + \dim_k R_{s-1} \} \\
    &= (s+2)(s+1)/2 - 1 - s(s+1)/2 \\
    &= s \\
\end{align*}
\endgroup
Choose a basis $\{ \phi_1 , \dots , \phi_s \}$ for $(I_1 \cap I_2)_s$. By Proposition \ref{alscomp}, $I_2$ defines a compressed ring, whence $\dim_k (I_2)_s = s+1$. Extend $\{\phi_1 , \dots , \phi_s \}$ to a basis $\{ \phi_1 , \dots , \phi_{s+1} \}$ of $(I_2)_s$. 

Observe that $(I_1 \cap I_2)_{s+1} = (I_2)_{s+1}$ by a dimension count. By Proposition \ref{btab1}, there exist $b$ linearly independent $s+1$-forms $\psi_1 , \dots  \psi_b$ such that 
$$(I_2)_{s+1} = (R_+ I_2)_{s+1} + \Span_k \{ \psi_1 , \dots , \psi_b \}.$$
Since $I$ defines a compressed ring, the degrees of its minimal generators are concentrated in $2$ consecutive degrees. This means that
$$I = (\phi_1 , \dots , \phi_s , \psi_1 , \dots , \psi_b) + R_+ \phi_{s+1}.$$
\end{proof}

\begin{cor}\label{evens}
Adopt Setup \ref{setup2}. Assume furthermore that the Betti table of $I_2$ is given by Proposition \ref{btab3}. If $s$ is even, then there exists a minimal generating set $\phi_1 , \dots , \phi_{s+1}$ for $I_2$ such that
$$I = (\phi_1 , \dots , \phi_s) + R_+ \phi_{s+1}$$
and this is a minimal generating set for $I$. $\qquad \square$
\end{cor}

\section{A Resolution for Certain Types of Ideals}\label{theres}

In view of Proposition \ref{trimmed}, our goal is to produce a resolution of ideals of the form $(\phi_1 , \dots , \phi_s , \psi_1 , \dots , \psi_b) + R_+ \phi_{s+1}$. We build the resolution in a more general setting, then specialize to the relevant case.

\begin{setup}\label{SU12}Let $R$ be a ring, $n$ be a positive integer, and $U$ and $V$ be free $R$-modules of rank $3$ and $2n+1$, respectively.
Suppose that $v_0\in V$ generates a free summand of rank 1. Let $w_0$ be an element of $V^*$ with $w_0(v_0)=1$ and let $V'$ be the kernel of $w_0$. Observe that
 \begin{equation}\label{12decomp}V=Rv_0\p V'.\end{equation} Let
\begin{equation}\label{phi}\phi=\phi'+(v_0\w v_0')\end{equation} be an element of $\bigwedge^2 V$, with $\phi'\in \bigwedge^2V'$ and $v_0'\in V'$, and let ${x:U\to R}$ be an $R$-module homomorphism.
Assume
\begin{enumerate}[\rm (a)]
\item\label{SU12.a} the ideal $\im (x:U\to R)$  has grade three,
\item\label{SU12.b} the ideal $\im \big(\phi^{(n)}:\bigwedge ^{2n}V^* \to R\big)$ is a grade three Gorenstein ideal,
\item\label{SU12.c} the ideal $\im(v_0':V^*\to R)$ is contained in the ideal $\im (x:U\to R)$, and
\item\label{SU12.d} the element $v_0\w \phi^{(n)}$ of $\bigwedge^{2n+1}V$ is regular.
\end{enumerate}
Let $\mcI$ be the ideal $$\ts \im \Big((V'\w \phi^{(n)}):\bigwedge ^{2n+1}V^* \to R\Big) +\im(x:U\to R)\cdot \im \Big(v_0\w \phi^{(n)}: \bigwedge ^{2n+1}V^* \to R\Big)$$of $R$.
\end{setup}

Let $K$ represent the ideal $\im \big(\phi^{(n)}:\bigwedge ^{2n}V^* \to R\big)$ of Setup \ref{SU12}.\ref{SU12.b}. This ideal has $2n+1$ generators: one for each element in a basis for $\bigwedge^{2n}V^*$. We have partitioned this generating set into two subsets and we consider the corresponding two subideals $K_0$ and $K'$ of $K$. The ideal
$$\ts K_0=\im \Big(v_0\w \phi^{(n)}: \bigwedge ^{2n+1}V^* \to R\Big)$$
 is principal and its generator is the Pfaffian (of the alternating matrix which corresponds to $\phi$) that involves all of the basis vectors  of $V'$.
The ideal
$$\ts K'=\im\Big((V'\w \phi^{(n)}):\bigwedge ^{2n+1}V^* \to R\Big)$$ has $2n$ generators and each of these generators is a Pfaffian (of the alternating matrix which corresponds to $\phi$) that involves $v_0$ together with $2n-1$ basis vectors from $V'$. Of course, $\mcI=K'+(\im x)\cdot K_0$

\begin{chunk}\label{12.4} Define $\proj: V\to V'$ to be the projection induced by the decomposition
(\ref{12decomp}); in other words, $\proj(v)=v-w_0(v)\cdot v_0$, for $v\in V$.
\end{chunk}

\begin{observation}\label{obs12.1} Adopt the terminology and hypotheses of
{\rm\ref{SU12}}. The following statements hold:
\begin{enumerate}
\item\label{obs12.1.y}$w_0(\phi)=v_0'\in V'$,
\item\label{obs12.1.z} $\phi^{(n)}={\phi'}^{(n)}+{\phi'}^{(n-1)}\w v_0\w v_0'$,
\item \label{obs12.1.b}there exists  an $R$-module homomorphism $q:V\to U$  for which the diagram
$$\xymatrix{&V^*\ar@{-->}[ld]_{\exists q}\ar[d]^{v_0'}\\U\ar[r]^x&\im x\ar[r]&0}$$ commutes.
\item \label{obs12.1.c}there exists an $R$-module homorphism $B: \bigwedge^{2n+1}V^*\to \bigwedge^2 U$
 for which the diagram
$$\xymatrix{
\bigwedge^{2n+1}V^*\ar[rr]
^{\phi^{(n)}
(\underline{\phantom{X}})}\ar[d]^{B}
&&V^*\ar[d]^{q}\\
\bigwedge^2U\ar[rr]^{x}&&
U
}$$commutes, and
\item\label{obs12.1.a}
$
K':_RK_0\subseteq \im(x: U\to R)$.
\end{enumerate}
\end{observation}
\begin{proof} Assertions (\ref{obs12.1.y}), (\ref{obs12.1.z}), and (\ref{obs12.1.b}) are evident.

To prove (\ref{obs12.1.c}), observe that hypothesis \ref{SU12}.\ref{SU12.a} ensures that \begin{equation}\label{12assum-May6}\ts 0\to \bigwedge^3 U\xrightarrow{x}
\bigwedge^2 U\xrightarrow{x}
  U\xrightarrow{x} R\end{equation} is an acyclic complex of free $R$-modules.
One computes:
\begin{align*}&(x\circ q)(\phi^{(n)}(w_{2n+1}))=v_0'((\phi^{(n)}(w_{2n+1}))
=(v_0'\w \phi^{(n)})(w_{2n+1})\\
=& (v_0'\w [{\phi'}^{(n)}+{\phi'}^{(n-1)}\w v_0\w v_0'])(w_{2n+1})
=(v_0'\w{\phi'}^{(n)})(w_{2n+1})=0,\end{align*} for each $w_{2n+1}\in \bigwedge^{2n+1}V^*$.
The first equality is (\ref{obs12.1.b}), the second equality  is the fact that $\bigwedge^\bullet V^*$ is a $\bigwedge^\bullet V$-module, the third  equality is (\ref{obs12.1.z}), and last equality holds because $v_0'\w {\phi'}^{(n)}$ is in $\bigwedge^{2n+1}{V'}=0$.) The assertion now follows from the acyclicity of (\ref{12assum-May6}).

\medskip\noindent (\ref{obs12.1.a})
The Buchsbaum-Eisenbud theorem \cite[Cor.~2.6, Thm.~3.1]{BE} guarantees that
{\begin{equation}\xymatrix{
&0\ar[r]&\bigwedge^{2n+1}V^*\ar[rr]
^{\phi^{(n) }
(\underline{\phantom{X}})}
&&V^*\ar[rr]^{
(\underline{\phantom{X}})
(\phi)
}&&V\ar[rr]^{(\underline{\phantom{X}})\w
\phi^{(n)}}&&\bigwedge^{2n+1}V}\label{BE}\end{equation}}
is a resolution of $R/K$. If $r\in R$ and  $rK_0\subseteq K'$, then there exists $v'\in V'$ with
$$v'\w \phi^{(n)} +rv_0 \w \phi^{(n)}=0.$$ The exactness of (\ref{BE}) guarantees that there exists an element $w\in V^*$ with $$w(\phi)=v'+rv_0.$$
Apply (\ref{phi}) to see that $$w(\phi') +w(v_0)\cdot v_0'-w(v_0')\cdot v_0=v'+rv_0.$$It follows that $[r+w(v_0')]\cdot v_0\in V'$; hence $r=v_0'(-w)\in \im v_0'$. The hypothesis \ref{SU12}.\ref{SU12.c} ensures that $\im v_0'\subseteq \im(x)$.
\end{proof}

\begin{theorem}\label{12.6} Adopt the terminology  and hypotheses of {\rm\ref{SU12}, \ref{12.4},} and {\rm\ref{obs12.1}}. Then the maps and modules
\begin{equation}\label{12.6.1}0\to \begin{matrix} \bigwedge^{2n+1}V^*\\\p\\\bigwedge^3 U\end{matrix}
\xrightarrow{\ \ d_3\ \ } \begin{matrix} V^*\\\p\\ \bigwedge^2 U\end{matrix}
\xrightarrow{\ \ d_2\ \ }\begin{matrix}V'\\\p\\U \end{matrix}
\xrightarrow{\ \ d_1\ \ } {\ts\bigwedge^{2n+1}V}\end{equation}
form a resolution of $R/\mcI$ by free $R$-modules,
where
$$d_3\begin{pmatrix} w_{2n+1}\\u_3\end{pmatrix}= \begin{pmatrix}\phi^{(n)}(w_{2n+1})\\B(w_{2n+1})+x(u_3)\end{pmatrix},$$
$$d_2\begin{pmatrix} w_{1}\\u_2\end{pmatrix}= \begin{pmatrix} \proj(w_1(\phi))\\-q(w_1)+x(u_2)\end{pmatrix},$$
and
$$d_1\begin{pmatrix} v'\\u_1\end{pmatrix} =(v'+x(u_1)\cdot v_0)\w \phi^{(n)},$$
for $w_{2n+1}\in \bigwedge^{2n+1}V^*$, $u_3\in \bigwedge^3U$, $w_1\in V^*$, $u_2\in \bigwedge^2U$, $v'\in V'$, and $u_1\in U$.
\end{theorem}
\begin{proof}The homomorphisms (\ref{12.6.1}) form the mapping cone of
\begin{equation}\label{12.6.2}\xymatrix{
&0\ar[rr]&&\bigwedge^{2n+1}V^*\ar[rr]
^{\phi^{(n)}
(\underline{\phantom{X}})}\ar[d]^{B}
&&V^*\ar[rr]^{\proj\Big(
(\underline{\phantom{X}})
(\phi)\Big)
}\ar[d]^{q}&&V'\ar[d]^{\underline{\phantom{X}} \w\phi^{(n)}}\\
0\ar[r]&\bigwedge^3U\ar[rr]^{x}&&
\bigwedge^2U\ar[rr]^{x}&&
U\ar[rr]^{x(\underline{\phantom{X}})\cdot v_0\w \phi^{(n)}}&&\bigwedge^{2n+1}V.
}\end{equation}
The  rows are complexes by (\ref{BE}) and (\ref{12assum-May6}). The left most square commutes by \ref{obs12.1}.(\ref{obs12.1.c}). To see that the right most square commutes, let $w\in V^*$. The clock-wise path sends $w$ to
\begin{align*}[w(\phi)-w_0(w(\phi))\cdot v_0]\w \phi^{(n)}
&=w(\phi)\w \phi^{(n)}+[w_0(\phi)](w)\cdot v_0\w \phi^{(n)}&&\text{by \ref{12.4}}\\
&=w(\phi^{(n+1)})+ v_0'(w)\cdot v_0\w \phi^{(n)}&&\text{by \ref{obs12.1}.(\ref{obs12.1.y})}\\
&=v_0'(w)\cdot v_0\w \phi^{(n)}.&&\end{align*}
The counter-clock-wise path sends $w$ to $$(x\circ q)(w)\cdot v_0\w \phi^{(n)}=v_0'(w)\cdot v_0\w \phi^{(n)}$$ by \ref{obs12.1}.(\ref{obs12.1.b}).

Apply the long exact sequence of homology associated to a mapping cone to see
 that the complex (\ref{12.6.1}) is acyclic.
It suffices to show that
\begin{enumerate}
\item\label{12.6.a} the top row of (\ref{12.6.2}) is a resolution of $K'/(K'\cap K_0)$,
\item\label{12.6.b} the bottom row of (\ref{12.6.2}) is a resolution of $R/(\im x)\cdot K_0$, and
\item\label{12.6.c}  the induced map on zero-th homology
$$\frac{K'}{K'\cap K_0} \to \frac {R}{(\im x) \cdot K_0}$$is an injection. This induced map is the following composition of
 natural maps
$$\xymatrix{ \frac{K'}{K'\cap K_0}\ar@{^(->}[r]&\frac{R}{K'\cap K_0}\ar@{->>}[r]&\frac{R}{(\im x)\cdot K_0}}.$$
(Recall from \ref{obs12.1}.(\ref{obs12.1.a}) that $(K'\cap K_0)\subseteq (\im x)\cdot K_0$.)
\end{enumerate}

\medskip\noindent {\it Proof of {\rm(\ref{12.6.a})}.} The augmented top row of (\ref{12.6.2}) is the bottom row of the following short exact sequence of complexes:
$$
\xymatrix{ &&0\ar[r]&Rv_0\ar[r]^{(\underline{\phantom{X}})\w
\phi^{(n)}}\ar@{^(->}[d]&K_0\ar@{^(->}[d]\ar[r]&0\\
0\ar[r]&\bigwedge^{2n+1}V^*\ar[r]\ar[d]^{=}&V^*\ar[r]\ar[d]^{=}&V\ar[r]^{(\underline{\phantom{X}})\w
\phi^{(n)}}\ar@{->>}[d]&K\ar[r]\ar@{->>}[d]&0\\
0\ar[r]&\bigwedge^{2n+1}V^*\ar[r]&V^*\ar[r]&V'\ar[r]^{(\underline{\phantom{X}})\w
\phi^{(n)}}&\frac{K'}{K'\cap K_0}\ar[r]&0
.}$$ The middle complex is exact because of \ref{BE}; the top complex is exact because of \ref{SU12}.\ref{SU12.d}.

Assertion (\ref{12.6.b}) is a consequence of (\ref{12assum-May6}) and \ref{SU12}.(\ref{SU12.d}). Assertion (\ref{12.6.c}) is immediate because $K'\cap (\im x)\cdot K_0\subseteq K'\cap K_0$. \end{proof}

\begin{remark}\label{translate}
Adopt the notation and hypotheses of Setup \ref{setup2}. Observe that in Setup \ref{SU12}, the map $x: U \to R$ is the first Koszul differential for the ideal $R_+$, so that $\im (x : U \to R) = R_+$. Similarly, choose $\phi$ such that
$$\im ( \phi^{(n)} : \bigwedge^{2n} V^* \to R ) = I_2.$$
The hypothesis \ref{SU12}.\ref{SU12.c} is simply the statement that $I_2$ is homogeneous and generated in positive degree. Similarly, since we are working over a polynomial ring (which is a domain), hypothesis \ref{SU12}.\ref{SU12.d} is trivially satisfied.

By Proposition \ref{trimmed}, there exists a minimal generating set $$(\phi_1 , \dots , \phi_{s+1} , \psi_1 , \dots , \psi_b)$$ for $I_2$ such that
$$I = (\phi_1 , \dots , \phi_s , \psi_1 , \dots , \psi_b) + R_+ \phi_{s+1}.$$
Choose $v_0$ to be the direct summand corresponding to the minimal generator $\phi_{s+1}$ of $I_2$; then, in the notation of Setup \ref{SU12}, 
$$\cat{I} = (\phi_1 , \dots , \phi_s , \psi_1 , \dots , \psi_b) + R_+ \phi_{s+1},$$
whence Theorem \ref{12.6} provides a resolution of the ideal $I$ as in Setup \ref{setup2}.
\end{remark}

\begin{cor}\label{numgenss}
In the notation and hypotheses of Theorem \ref{12.6}, assume that $(R,\m , k)$ is a local ring or $R$ is a standard graded polynomial ring over a field $k$. Then
$$\mu ( \cat{I} ) = \mu (K) + 2 - \rank_k (q \otimes k)$$
\end{cor}

\begin{proof}
Observe that $d_1 \otimes k = 0$ by construction. Let $\cat{F}$ denote the complex of Theorem \ref{12.6}. Then,
$$ H_1 (\cat{F} \otimes k ) = \frac{(V' \oplus U) \otimes k}{\im (d_2 \otimes k)}.$$
The only nonzero entries of $d_2 \otimes k$ come from $q \otimes k$, since the other differentials in the mapping cone are built from minimal resolutions. Thus
$$\rank_k (d_2 \otimes k) = \rank_k ( q \otimes k)$$
and
\begin{align*}
    \dim_k H_1 ( \cat{F} \otimes k) &= \dim_k (V' \oplus U)\otimes k - \rank_k (q \otimes k) \\
    &= \mu(K) + 2 - \rank_k (q \otimes k). \\
\end{align*}
To conclude, recall that $\dim_k H_1 ( \cat{F} \otimes k) = \dim_k \tor_1^R (R/\cat{I} , k) = \mu(\cat{I})$.
\end{proof}

\begin{cor}\label{evens2}
Adopt Setup \ref{setup2}, where $s$ is even. Assume $I_2$ has Betti table given by Proposition \ref{btab3}. Under the identifications of Remark \ref{translate}, $I$ has homogeneous minimal free resolution of the form
$$0 \to \begin{matrix} R(-2s-2) \\ \p \\ R(-s-3) \end{matrix} \to R(-s-2)^{s+4} \to \begin{matrix} R(-s)^s \\ \p \\ R(-s-1)^3 \\ \end{matrix} \to R.$$
In particular, the Betti table for $R/I$ as an $R$-module is:
$$\begin{tabular}{L|L|L|L|L}
     & 0 & 1 & 2 & 3  \\
     \hline 
   0  & 1 & 0 & 0 & 0 \\
   \hline
   s-1 & 0 & s & 0 & 0 \\
   \hline 
   s & 0 & 3 & s+4 & 1 \\
   \hline 
   2s-1 & 0 & 0 & 0 &1 \\
\end{tabular}$$ 
\end{cor}

\begin{proof}
This is an immediate consequence of Corollary \ref{evens} combined with the resolution of Theorem \ref{12.6}.
\end{proof}

\section{Applications and Examples}\label{appex}

In this section we examine classes of examples arising from variants of matrices defined in Section $3$ of \cite{christ}. 
\begin{definition}\label{Umat}
Let $U_m^{ev}$ denote the $m\times m$ matrix with entries from the polynomial ring $R=k[x,y,z]$ defined by:
$$U^{ev}_{i,m-i} = x^2, \quad U^{ev}_{i,m-i+1} = z^2, \quad U^{ev}_{i,m-i+2} = y^2.$$
Similarly, define $U_m^{odd}$ via:
$$U^{odd}_{i,m-i} = x^2, \quad U^{odd}_{i,m-i+1} = z^2, \quad U^{odd}_{i,m-i+2} = y^2, \ \textrm{for } i<m$$
and $U^{odd}_{m,1} = z$, $U^{odd}_{m,2} = y$. All other entries are defined to be $0$. Define $d_m^{ev} := \det ( U_m^{ev})$ and $d_m^{odd} := U_{m}^{odd}$.
\end{definition}

\begin{observation}
For all $i = 1 , \dots , m$,
$$U_m^{ev} = \begin{pmatrix} O_{x^2} & U^{ev}_i \\
U^{ev}_{m-i+1} & ^{y^2}O \\
\end{pmatrix}$$
$$U_m^{odd} = \begin{pmatrix} O_{x^2} & U^{ev}_i \\
U^{odd}_{m-i+1} & ^{y^2}O \\
\end{pmatrix}$$
where $O_{x^2}$ denotes the appropriately sized matrix with $x^2$ as the bottom rightmost corner entry and zeroes elsewhere. Similarly, $^{y^2}O$ denotes the appropriately sized matrix with $y^2$ in the top leftmost corner and zeroes elsewhere.
\end{observation}

\begin{definition}
Define $V_m^{ev}$ to be the $(2m+1)\times (2m+1)$ skew symmetric matrix
$$V^{ev}_m := \begin{pmatrix}
O & O_{x^2} & U_m^{ev} \\
-(O_{x^2})^T & 0 & ^{y^2}O \\
-U_m^{ev} & -(^{y^2}O)^T & O \\
\end{pmatrix}$$
and $V_m^{odd}$ to be the $(2m+1) \times (2m+1)$ skew symmetric matrix 
$$V^{odd}_m := \begin{pmatrix}
O & O_{x^2} & (U_m^{odd})^T \\
-(O_{x^2})^T & 0 & ^{y^2}O \\
-U_m^{odd} & -(^{y^2}O)^T & O \\
\end{pmatrix}$$
\end{definition}

To see the pattern a little more clearly, the first first couple of matrices are:
$$V_1^{ev} = \begin{pmatrix}
      0&x^{2}&z^{2}\\
      {-x^{2}}&0&y^{2}\\
      {-z^{2}}&{-y^{2}}&0\end{pmatrix}, \quad V_2^{ev} = \begin{pmatrix}
      0&0&0&x^{2}&z^{2}\\
      0&0&x^{2}&z^{2}&y^{2}\\
      0&{-x^{2}}&0&y^{2}&0\\
      {-x^{2}}&{-z^{2}}&{-y^{2}}&0&0\\
      {-z^{2}}&{-y^{2}}&0&0&0\end{pmatrix}$$
$$V_1^{odd} = \begin{pmatrix}
      0&x^{2}&z\\
      {-x^{2}}&0&y\\
      {-z}&{-y}&0\end{pmatrix}, \quad V_2^{odd} = \begin{pmatrix}
      0&0&0&x^{2}&z\\
      0&0&x^{2}&z^{2}&y\\
      0&{-x^{2}}&0&y^{2}&0\\
      {-x^{2}}&{-z^{2}}&{-y^{2}}&0&0\\
      {-z}&{-y}&0&0&0\end{pmatrix}$$

\begin{prop}
The ideal of submaximal pfaffians $\textrm{Pf} (V_m^{ev})$ is minimally generated by the elements
$$x^{2m-2i}d_i^{ev}, \quad y^{2m-2i}d_i^{ev} \ \textrm{for} \ 0 \leq i \leq m-1, \ \textrm{and} \ d_m^{ev}$$
Similarly, the ideal of submaximal pfaffians $\textrm{Pf} (V_m^{odd})$ is minimally generated by the elements
$$x^{2m-2i} d_i^{odd} \ \textrm{for} \ 0 \leq i \leq m-1 \ y^{2m-1}, \ y^{2m-2i}d_i^{odd} \ \textrm{for} \ 0 \leq i \leq m-2$$
$$\textrm{and} \ d_m^{odd}$$
\end{prop}

\begin{proof}
The proof is essentially identical to that of Proposition $3.3$ in \cite{christ}. 
\end{proof}

\begin{example}
Consider the ideal $\textrm{Pf} (V_m^{odd})$ and trim the generator $x^{2m-2}d_1^{odd} = x^{2m-2}z$. Let $\phi \in \bigwedge^2 V$ denote the element corresponding to the matrix $V_m^{odd}$. If we consider $x^{2m-2}z$ as the $2m$th generator, decompose $V = \bigoplus_{i=1}^{2m+1} Re_i$ as $V' \oplus Re_{2m}$, where the basis $e_i$ corresponds to the $i$th signed submaximal pfaffian of $V_m^{odd}$. Let $\phi \in \bigwedge^2 V$ denote the element corresponding to the matrix $V_m^{odd}$; notice
$$\phi = \phi' -e_{2m} \w (x^2e_1 + z^2 e_2 + y^2 e_3)$$
for some $\phi' \in \bigwedge^2 V'$. Take $v_0' := -x^2 e_1 - z^2 e_2 - y^2 e_3$ and let $U = Re_x \oplus Re_y \oplus Re_z$ with map $X : e_x \mapsto x$, $e_y \mapsto y$, and $e_z \mapsto z$. If $q : V^* \to U$ is the map sending $e_1^* \mapsto -xe_x$, $e_2^* \mapsto -z e_z$, $e_3^* \mapsto -y e_y$, and all other basis element to $0$, then the following diagram commutes:
$$\xymatrix{ & V^* \ar[dl]_-{q} \ar[d]^-{v_0'} \\
U \ar[r]^{X} & \im X \\}$$
Similarly, if $B : \bigwedge^{2m+1} V^* \to \bigwedge^2 U$ is defined by sending $\omega \mapsto x y^{2m-2} e_x \w e_y + y^{2m-3} z^2 e_y \w e_z$ (where $\omega$ is a generator for $\bigwedge^{2m+1} V^*$), then the following diagram commutes:
$$\xymatrix{\bigwedge^{2m+1} V^* \ar[r]^-{\phi^{(n)} (-)} \ar[d]^-{B} & V^* \ar[d]^-{q} \\
\bigwedge^2 U \ar[r]^-{X} & U \\}$$
Employing the construction of Theorem \ref{12.6}, we deduce that the mapping cone is acyclic; moreover, the entries of all maps involved have entries in $R_+$, so this is a minimal free resolution of the ideal
$$J = (\textrm{Pf} ( V_m^{odd}) \backslash  \textrm{Pf}_{2m} (V_m^{odd}) ) + R_+\textrm{Pf}_{2m} (V_m^{odd}).$$
This implies that $J$ is minimally generated by $2m+3$ elements (and by Lemma \ref{tormins} defines a ring of Tor algebra class $G(2m)$). 
\end{example}

\begin{prop}\label{isminl}
Let $R=k[x,y,z]$ with the standard grading, where $k$ is any field. Let $I \subset R_+^2$ be a homogeneous $R_+$-primary grade $3$ Gorenstein ideal with generators obtained from the ideal of submaximal pfaffians $\textrm{Pf} (M)$ for some skew symmetric matrix $M$. If the $i$th row of $M$ has entries in $R_{>1}$, then 
$$(\textrm{Pf}_j (M) \mid j \neq i ) + R_+ \textrm{Pf}_i (M)$$
is minimally generated by $\mu(I)+2$ elements and defines a ring of type $2$.
\end{prop}

\begin{proof}
Let $M$ be obtained from the element $\phi \in \bigwedge^2 V$, where $V$ is a free module of rank $\mu (I)$. Let $e_i$ be the basis element of $V$ corresponding to $\textrm{Pf}_i (M)$; write $V = V' \oplus Re_i$. It must be verified that $q \otimes k=0$ and $B \otimes k =0$. The statement that the $i$th row of $M$ has entries in $R_{>1}$ means that if
$$\phi = \phi' + e_i \w v_0',$$
then $0 \neq v_0' \in R_{>1} V$. Counting degrees on the induced map $q : V^* \to U$ of Observation \ref{obs12.1}.\ref{obs12.1.z}, we deduce $q (V^*) \subset R_+ U$, so that $q \otimes k = 0$. In particular, by Proposition \ref{numgenss},
$$(\textrm{Pf}_j (M) \mid j \neq i ) + R_+ \textrm{Pf}_i (M)$$
is minimally generated by $\mu (I) - 1 + 3 = \mu(I)+2$ elements.

By the assumption that $I \subseteq R_+^2$, a degree count shows that $B \otimes k = 0$. More precisely, if $t$ is the initial degree of $I$, then the matrix representation of $B$ with respect to the standard bases has entries in $R_+^{t-1}$. We conclude that the resolution provided by Theorem \ref{12.6} is minimal.
\end{proof}

\section{A Class of Ideals with Extremal Graded Betti Numbers}\label{extbetti}

Adopt Setup \ref{setup2}. The ideal $I_2$ has Betti table arising from Proposition \ref{btab1} for some integer $b < s+1$ by Proposition \ref{alscomp}. We may fit $I$ into the short exact sequence
$$0 \to I_2 / I \to R/I \to R/I_2 \to 0$$
whence upon counting ranks on the graded strands of the long exact sequence of $\tor$, we deduce that $\dim_k \tor_1^R (R/I,k)_{s+1} \leq b+3$. Since $b \leq s$, $\dim_k \tor_1^R (R/I,k)_{s+1} \leq s+3$. Furthermore: 
\begin{prop}\label{torbound}
Let $I$ be as in Setup \ref{setup2}. Then $\dim_k \tor_1^R (R/I,k)_{s+1} \leq s+2$. 
\end{prop}
 
 \begin{proof}
 By counting ranks on the long exact sequence of $\tor$ induced by the short exact sequence
 $$0 \to I_2 / I \to R/I \to R/I_2 \to 0,$$
 we must have $\dim_k \tor_1^R (R/I_2,k)_{s+1} = s$, which is the maximum possible. By Proposition \ref{trimmed}, $I$ may be written
 $$I = (\phi_1 , \dots , \phi_s , \psi_1, \dots , \psi_s) + R_+ \phi_{s+1}$$
 where $\phi_1 , \dots , \phi_{s+1} , \psi_1, \dots , \psi_s$ is a minimal generating set for $I_2$. Assume for sake of contradiction that $\dim_k \tor_1^R (R/I,k)_{s+1} = s+3$; this means that the above generating set for $I$ is minimal. Thus the resolution of Theorem \ref{12.6} is a minimal free resolution for $R/I$.
 
 Let us examine the map $q : V^* \to U$. By counting degrees, one finds that $q(V^*) \not\subset R_+ U$ or that $q$ is identically the $0$ map. Either case is a contradiction, so that $\rank_k (q \otimes k) \geq 1$. 
 \end{proof}
 
 We now exhibit a class of ideals defining compressed rings with top socle degree $2s-1$ and $\dim_k \tor_1^R (R/I,k)_{s+1} = s+2$, showing that the inequality of \ref{torbound} is sharp. To do this, we first need some notation.

\begin{definition}
Let $U_m^{j}$ (for $j \leq m$) denote the $m\times m$ matrix with entries from the polynomial ring $R=k[x,y,z]$ defined by:
$$(U^{j}_m)_{i,m-i} = x^2, \quad (U^{j}_m)_{i,m-i+1} = z^2, \quad (U^{j}_m)_{i,m-i+2} = y^2 \ \textrm{for} \  i \leq m-j$$
$$(U^{j}_m)_{i,m-i} = x, \quad (U^{j}_m)_{i,m-i+1} = z, \quad (U^{j}_m)_{i,m-i+2} = y \ \textrm{for} \  i >m-j$$
and all other entries are defined to be $0$. Define $d_m^j := \det ( U_m^j)$. 
\end{definition}

To see the pattern, we have:
$$U_2^1 = \begin{pmatrix}
       x^{2}&z^{2}\\
       z&y\end{pmatrix}, \ U_3^1 = \begin{pmatrix}
       0&x^{2}&z^{2}\\
       x^{2}&z^{2}&y^{2}\\
       z&y&0\end{pmatrix}, \ U_3^2 = \begin{pmatrix}
       0&x^{2}&z^{2}\\
       x&z&y\\
       z&y&0\end{pmatrix}$$
\begin{definition}
Define $V_m^{j}$ (for $j< m$) to be the $(2m+1)\times (2m+1)$ skew symmetric matrix
$$V^{j}_m := \begin{pmatrix}
O & O_{x^2} & (U_m^{j})^T \\
-(O_{x^2})^T & 0 & ^{y^2}O \\
-U_m^{j} & -(^{y^2}O)^T & O \\
\end{pmatrix}$$
and if $j=m$, then $V_m^m$ is the skew symmetric matrix
$$V^{j}_m := \begin{pmatrix}
O & O_{x^2} & (U_m^{m})^T \\
-(O_{x^2})^T & 0 & ^{y}O \\
-U_m^{m} & -(^{y}O)^T & O \\
\end{pmatrix}$$
\end{definition}

Observe that the ideal of pfaffians $\textrm{Pf} (V_m^j)$ is a grade $3$ Gorenstein ideal with graded Betti table
$$\begin{tabular}{L|L|L|L|L}
     & 0 & 1 & 2 & 3  \\
     \hline 
   0  & 1 & 0 & 0 & 0 \\
   \hline
   2m-j-1 & 0 & 2m+1-j & j & 0 \\
   \hline 
   2m-j & 0 & j & 2m+1-j & 0 \\
   \hline 
   4m-2j-1 & 0 & 0 & 0 &1 \\
\end{tabular}$$ 
In particular, for any integer $s$, $\textrm{Pf} (V_s^s)$ has Betti table
$$\begin{tabular}{L|L|L|L|L}
     & 0 & 1 & 2 & 3  \\
     \hline 
   0  & 1 & 0 & 0 & 0 \\
   \hline
   s-1 & 0 & s+1 & s & 0 \\
   \hline 
   s & 0 & s & s+1 & 0 \\
   \hline 
   2s-1 & 0 & 0 & 0 &1 \\
\end{tabular}$$
\begin{chunk}\label{pfnote}
Given an $n \times n$ alternating matrix $M$, the notation $\textrm{Pf} (M)$ will denote the ideal of submaximal pfaffians of the matrix $M$. Similarly,
$$(\textrm{Pf} (M) \backslash \textrm{Pf}_i (M))$$
is shorthand for the ideal
$$(\textrm{Pf}_j (M) \mid 1 \leq j \leq n, \ j \neq i ),$$
where $\textrm{Pf}_j (M)$ denotes the pfaffian of the matrix obtained by deleting the $j$th row and column of $M$.
\end{chunk}
\begin{prop}\label{maxideal}
Let $s \geq 1$ be an integer and $R =k[x,y,z]$, where $k$ is any field. The ideal 
$$I := (\textrm{Pf} (V_s^s) \backslash \textrm{Pf}_{s+1} (V_s^s)) + R_+ \textrm{Pf}_{s+1} (V_s^s)$$
is minimally generated by $2s+2$ elements and defines a compressed ring with $\soc (R/I ) \cong k(-s) \oplus k(-2s+1)$ and Betti table
$$\begin{tabular}{L|L|L|L|L}
     & 0 & 1 & 2 & 3  \\
     \hline 
   0  & 1 & 0 & 0 & 0 \\
   \hline
   s-1 & 0 & s & s-1 & 0 \\
   \hline 
   s & 0 & s+2 & s+4 & 1 \\
   \hline 
   2s-1 & 0 & 0 & 0 &1 \\
\end{tabular}$$
\end{prop}

\begin{proof}
We again use the resolution provided by Theorem \ref{12.6}. Let $V = \bigoplus_{i=1}^{2s+1} Re_i$, where $e_i$ corresponds to the $i$th pfaffian of $V_s^s$. Write $V = V' \oplus Re_{s+1}$ and let $\phi \in \bigwedge^2 V$ denote the element corresponding to $V^s_s$. By definition of the matrix $V_s^s$, we see
$$\phi = \phi' + e_{s+1} \w (-x^2 e_s + ye_{s+2}).$$
Set $v_0' := -x^2 e_s + ye_{s+2}$ and let $U = Re_x \oplus Re_y \oplus Re_z$ with map $X : e_x \mapsto x$, $e_y \mapsto y$, and $e_z \mapsto z$. Let $q :V^* \to U$ be the map sending $e_{s}^* \mapsto -x e_x$, $e_{s+2}^* \mapsto e_y$, and all other basis vectors to $0$. Then the following diagram commutes:
$$\xymatrix{ & V^* \ar[dl]_-{q} \ar[d]^-{v_0'} \\
U \ar[r]^{X} & \im X \\}$$
In particular, $\rank_k (q \otimes k) = 1$. We do not have to compute the map $B : \bigwedge^{2s+1} V^* \to \bigwedge^2 U$, since a degree count tells us that $B \otimes k = 0$. Thus the resolution of Theorem \ref{12.6} is not minimal, but we deduce that we only need to take the quotient by a subcomplex of the form 
$$0 \to R(-s-1) \to R(-s-1) \to 0$$
to obtain a minimal resolution. This immediately yields the Betti table of the statement; the other claims are immediate consequences of the Betti table. 
\end{proof}

\section{Tor Algebra Structures}\label{toralgstr}

In this section we examine some consequences for the Tor-algebra structures of the ideals resolved by Theorem \ref{12.6}. We start by defining what it means to have Tor algebra class $G(r)$. Although there are other Tor algebra families, we will only concern ourselves with the class $G$; the other families with their definitions may be found in \cite[Theorem 2.1]{torclass}.

\begin{definition}\label{classg}
Let $(R,\m,k)$ be a regular local ring with $I\subset \m^2$ and ideal such that $\pd_R (R/I) = 3$. Let $T_\bullet := \tor_\bullet^R (R/I , k)$. Then $R/I$ has Tor algebra class $G(r)$ if, for $m = \mu(I)$ and $t = \type (R/I)$, there exist bases for $T_1$, $T_2$, and $T_3$
$$e_1 , \dots , e_m, \quad f_1 , \dots , f_{m+t-1} , \quad g_1 , \dots , g_t,$$
respectively, such that the only nonzero products are given by
$$e_i f_i = g_1 = f_i e_i, \quad 1 \leq i \leq r.$$
Such a Tor algebra structure has
$$T_1 \cdot T_1 = 0, \quad \rank_k (T_1 \cdot T_2 ) = 1, \quad \rank_k (T_2 \to \hom_k (T_1 , T_3) ) = r,$$
where $r \geq 2$.
\end{definition}

\begin{theorem}[\cite{christ}, Theorem $2.4$, Homogeneous version]\label{toralg}
Let $R = k[x,y,z]$ with the standard grading and let $I \subseteq R_+^2$ be an $R_+$-primary homogeneous Gorenstein ideal minimally generated by elements $\phi_1 , \dots , \phi_{2m+1}$. Then the ideal
$$J = (\phi_1 , \dots , \phi_{2m} ) + R_+\phi_{2m+1}$$
is a homogeneous $R_+$-primary ideal and defines a ring of type $2$. Moreover,
\begin{enumerate}[(a)]
    \item If $m=1$, then $\mu (J) = 5$ and $R/J$ is class $B$.
    \item If $m=2$, then one of the following holds:
    \begin{equation*}
    \begin{split}
        &\bullet \mu(J) = 4 \ \textrm{and} \ R/J \ \textrm{is class} \ H(3,2) \\
        &\bullet \mu(J) = 5 \ \textrm{and} \ R/J \ \textrm{is class} \ B \\
        &\bullet \mu(J) \in \{ 6,7 \} \ \textrm{and} \ R/J \ \textrm{is class} \ G(r) \ \textrm{with} \ \mu(J)-2 \geq r \geq \mu(J)-3 \\
        \end{split}
    \end{equation*}
    \item If $m\geq 3$, then $R/J$ is class $G(r)$ with $\mu(J)-2 \geq r \geq \mu(J) - 3$. 
\end{enumerate}
\end{theorem}

\begin{prop}
Adopt Setup \ref{setup2}. Assume furthermore that the Betti table of $I_2$ is given by those of Proposition \ref{btab3}. If $s$ is even, then $R/I$ defines a ring of Tor algebra class $G(s)$.
\end{prop}
\begin{proof}
By Theorem \ref{toralg} combined with Proposition \ref{evens}, $I$ is class $G(r)$ for some $r \geq s$. It suffices to show that the induced map on the Tor algebra
$$\delta : T_2 \to \hom (T_1 , T_3)$$
has rank $\leq s$, where $T_i := \tor_i^R (R/I , R )$. Examining the Betti table of Corollary \ref{evens2}, we see that
$$(T_1)_{s+1} \cdot (T_2) \subseteq (T_2)_{\geq 2s+3} = 0,$$
whence the only nontrivial products can occur between $(T_1)_s$ and $T_2$, implying $\rank_k \delta \leq s$. 
\end{proof}

\begin{prop}\label{propmax}
Adopt notation and hypotheses of Proposition \ref{maxideal}. Then $I$ defines a ring of Tor algebra class $G(2s-1)$. 
\end{prop}

\begin{proof}
If $s=2$, a direct computation in Macaulay2 verifies that $I$ is class $G(3)$, so assume that $s \geq 3$. By Theorem \ref{toralg}, $I$ is class $G(r)$ for some $r \geq 2s-1$. It remains to show
$$\delta : T_2 \to \hom (T_1 , T_3)$$
has rank $\leq 2s-1$, where $T_i := \tor_i^R (R/I , R )$. First, observe that
$$(T_1)_s \cdot (T_2)_{s+1} \subseteq (T_2)_{2s+1}$$
and $(T_3)_{2s+1} = 0$ since $s \geq 3$. Similarly,
$$(T_1)_{s+1} \cdot (T_2)_{s+2} \subseteq (T_3)_{2s+3} = 0$$
whence the only nontrivial products are between $(T_1)_s$, $(T_2)_{s+2}$ and $(T_1)_{s+1}$, $(T_2)_{s+1}$, implying
$$\rank_k \delta \leq s + (s-1) = 2s-1.$$
\end{proof}

\begin{cor}\label{rbounds}
Adopt Setup \ref{setup2} with $s \geq 3$. Then $I$ has Tor algebra class $G(r)$ for some $s \leq r \leq 2s-1$.
\end{cor}

\begin{proof}
Observe that $\mu(I) \geq s+3$. Moreover, $\mu (I) \leq 2s+2$ by Proposition \ref{torbound}. If $\mu (I) = 2s+2$, then $I$ has Betti table identical to that of the ideal in Proposition \ref{maxideal}, so that by the proof of Proposition \ref{propmax}, $I$ has Tor algebra class $G(2s-1)$. 
\end{proof}

\begin{lemma}\label{tormins}
Adopt Setup \ref{setup2} with $s \geq 3$. Then $I$ defines a compressed ring of Tor algebra class $G ( \mu(I)-3)$.
\end{lemma}

\begin{proof}
Consider the degree $s+1$ strand of the long exact sequence of Tor associated to the short exact sequence
$$0 \to \frac{I_2}{I} \to \frac{R}{I} \to \frac{R}{I_2} \to 0.$$
We obtain:
\begin{align*}
    0 &\to \tor_2^R (R/I , k )_{s+1} \to \tor_2 (R/I_2 , k )_{s+1} \\
    &\to \tor_1 (I_2 / I , k)_{s+1} \to \tor_1^R (R/I , k)_{s+1} \\
    &\to \tor_1^R (R/I_2 , k)_{s+1} \to 0. \\
\end{align*}
Counting ranks,
\begin{align*}
    \dim_k \tor_2^R (R/I , k)_{s+1} &= b -3+ \mu(I) - s -b \\
    &= \mu(I) - s - 3. \\
\end{align*}
A similar, but easier, rank count on the degree $s+2$ strand yields 
$$\dim_k \tor_2^R (R/I , k)_{s+2} = s+4,$$
implying $R/I$ has Betti table
$$\begin{tabular}{L|L|L|L|L}
     & 0 & 1 & 2 & 3  \\
     \hline 
   0  & 1 & 0 & 0 & 0 \\
   \hline
   s-1 & 0 & s & \mu(I)-s-3 & 0 \\
   \hline 
   s & 0 & \mu(I) - s & s+4 & 1 \\
   \hline 
   2s-1 & 0 & 0 & 0 &1 \\
\end{tabular}$$
By Theorem \ref{toralg} combined with Proposition \ref{trimmed}, the ideal $I$ defines a ring of Tor algebra class $G(r)$ for $r \geq \mu(I)- 3$. We examine the induced map
$$\delta : T_2 \to \hom (T_1 , T_3),$$
where $T_i := \tor_i^R (R/I , R )$. Notice that the only nontrivial products can occur between $(T_1)_{s}$, $(T_2)_{s+2}$ and $(T_1)_{s+1}$, $(T_2)_{s+1}$, implying that
$$\rank_k \delta \leq s + ( \mu(I) - s - 3 ) = \mu(I) - 3$$
\end{proof}

A natural question arising from Corollary \ref{rbounds} is whether or not every possible $r$ value may be obtained for a given $s \geq 3$, where $I$ is obtained from Setup \ref{setup2}. The next proposition will allow us to answer in the affirmative:
\begin{prop}\label{torach}
Let $s \geq 3$ be an integer and $R =k[x,y,z]$, where $k$ is any field. 
\begin{enumerate}
    \item For $1\leq i < s/2$, the ideal 
$$I := (\textrm{Pf} (V_{s-i}^{s-2i}) \backslash \textrm{Pf}_{s-i+1} (V_{s-i}^{s-2i})) + R_+ \textrm{Pf}_{s-i+1} (V_{s-i}^{s-2i})$$
has Tor algebra class $G(2s-2i)$.
\item For $1 \leq i < s/2$, the ideal
$$I := (\textrm{Pf} (V_{s-i}^{s-2i}) \backslash \textrm{Pf}_{i+1} (V_{s-i}^{s-2i})) + R_+ \textrm{Pf}_{i+1} (V_{s-i}^{s-2i})$$
has Tor algebra class $G(2s-2i-1)$.
\end{enumerate} 
\end{prop}

\begin{proof}
In view of Corollary \ref{numgenss} and Lemma \ref{tormins}, it suffices to compute the rank of the map $q : V^* \to U$ as in Theorem \ref{12.6} to find the minimal number of generators; consequently,
\begin{align*}
    \dim_k \tor_1^R (R/I , k)_{s+1} &= \mu (I) - s \\
    &= \mu (\textrm{Pf} (V_{s-i}^{s-2i}) ) + 2 - \rank_k (q \otimes k) - s  \\
    &= s+3-2i - \rank_k ( q \otimes k) \\
\end{align*}
The above equality follows from a rank count on the long exact sequence of $\tor$ associated to the short exact sequence
$$0 \to \frac{\textrm{Pf} (V_{s-i}^{s-2i})}{I} \to \frac{R}{I} \to \frac{R}{\textrm{Pf} (V_{s-i}^{s-2i})} \to 0$$
combined with the fact that $\dim_k \tor_1^R ( R/\textrm{Pf} (V_{s-i}^{s-2i}) , k)_{s} = s+1$. 

We compute the map $q$ explicitly in each case. Let $\phi \in \bigwedge^2 V$ represent the matrix $V_{s-i}^{s-2i}$; in the first case, after writing $V = V' \oplus Re_{s-i+1}$, we see (recalling that $i>0$)
$$\phi = \phi' + e_{s-i+1} \w ( -x^2 e_{s-i} + y^2 e_{s-i+2} ).$$
Let $U = Re_x \oplus Re_y \oplus Re_z$ with map $X : e_x \mapsto x$, $e_y \mapsto y$, and $e_z \mapsto z$. Take $q : V^* \to U$ to be the map sending $e_{s-i}^* \mapsto -x e_x$, $e_{s-i+2}^* \mapsto ye_y$, and all other basis vectors map to $0$. Clearly $q \otimes k = 0$, whence the resolution of Theorem \ref{12.6} is minimal. In particular, $\mu (I) = 2s+3-2i$.

For the second case, retain much of the notation as above. Decompose $V = V' \oplus Re_{i+1}$ and write 
$$\phi = \phi' + e_{i+1} \w (x^2 e_{2s-i} + z^2e_{2s-i+1} + ye_{2s-i+2} ).$$
Take $q : V^* \to U$ to be the map sending $e_{2s-i}^* \mapsto xe_x$, $e_{2s-i+1}^* \mapsto ze_z$, $e_{2s-i+2}^* \mapsto e_y$, and all other basis vectors to $0$. In this case $\rank_k (q \otimes k) = 1$, whence $\mu (I) = 2s+2-2i$.
\end{proof}

\begin{cor}\label{torach2}
Let $R = k[x,y,z]$ with the standard grading, where $k$ is any field. Given any $s \geq 3$ and any $r$ with $s \leq r \leq 2s-1$, there exists an ideal $I$ with $\soc (R/I) = k(-s) \oplus k(-2s+1)$ and defining an Artinian compressed ring of Tor algebra class $G(r)$. 
\end{cor}

\begin{proof}
Assume first that $r$ is even and $s \leq r < 2s-1$. In the case $r=s$, use Proposition \ref{evens2} on the ideal $( \textrm{Pf} (V_s^{ev} ) \backslash \textrm{Pf}_1 (V_s^{ev} )) + R_+  \textrm{Pf}_1 (V_s^{ev} )$. If $r>s$, employ Proposition \ref{torach} on the ideal
$$(\textrm{Pf} (V_{r/2}^{r-s}) \backslash \textrm{Pf}_{r/2+1} (V_{r/2}^{r-s})) + R_+ \textrm{Pf}_{r/2+1} (V_{r/2}^{r-s})$$

Assume now that $r$ is odd, with $s \leq  r \leq 2s-1$. If $r=2s-1$, use the ideal from Proposition \ref{propmax}. If $r<2s-1$, apply Proposition \ref{torach} to the ideal
$$(\textrm{Pf} (V_{(r+1)/2}^{r+1-s}) \backslash \textrm{Pf}_{s-(r+1)/2+1} (V_{(r+1)/2}^{r+1-s})) + R_+ \textrm{Pf}_{s-(r+1)/2+1} (V_{(r+1)/2}^{r+1-s})$$
\end{proof}

\section{Socle Minimally Generated in Degrees $s$, $2s-2$}\label{evencase}

In this section we further exploit properties of the resolution of Theorem \ref{12.6}.

\begin{setup}\label{setup3}
Let $k$ be a field and $V$ a $k$-vector space of dimension $3$; view $S(V)$ as graded by the standard grading. Let $I \subset R := S(V)$ be a grade $3$ homogeneous ideal defining a compressed ring with $\soc (R/I ) = k(-s) \oplus k(-2s+2)$, where $s \geq 3$. 

Write $I = I_1 \cap I_2$ for $I_1$, $I_2$ homogeneous grade $3$ Gorenstein ideals defining rings with socle degrees $s$ and $2s-2$, respectively. The notation $R_+$ will denote the irrelevant ideal ($R_{>0}$).
\end{setup}

It turns out that the ideals of Setup \ref{setup3} are also resolved by Theorem \ref{12.6}.

\begin{prop}\label{alscomp2}
Adopt Setup \ref{setup3}. Then the ideal $I_2$ defines a compressed ring.
\end{prop}

\begin{proof}
This is identical to the proof of Proposition \ref{alscomp}.
\end{proof}

\begin{lemma}\label{evenbtab}
Adopt Setup \ref{setup3}. Then $R/I_2$ has Betti table
$$\begin{tabular}{L|L|L|L|L}
     & 0 & 1 & 2 & 3  \\
     \hline 
   0  & 1 & 0 & 0 & 0 \\
   \hline
   s-1 & 0 & 2s+1 & 2s+1 & 0 \\
   \hline 
   s & 0 & 0 & 0 & 0 \\
   \hline 
   2s-2 & 0 & 0 & 0 &1 \\
\end{tabular}$$
\end{lemma}

\begin{proof}
We employ Proposition \ref{equalities}, where $r=3$, $c = 2s-2$, $m=1$, and $t = s$ ($= \lceil (2s-2)/2 \rceil$; see Proposition \ref{proplol}). Using the notation
$$T_i := \tor_i^R (R/I_2,k),$$
we obtain
\begin{equation*}
    \begin{split}
       & \dim (T_1)_{s} = 2s+1 \\
       &\dim (T_2)_{s+1} - \dim (T_1)_{s+1} = 2s+1 \\
       &\dim (T_2)_{s+2} = 0 \\
    \end{split}
\end{equation*}
Observe that we must have $\dim_k (T_1)_{s+1} = 0$, since otherwise the resolution of $R/I_2$ would not be self-dual, contradicting the fact that $I_2$ is Gorenstein. This yields the result.
\end{proof}

\begin{prop}\label{genset}
Adopt Setup \ref{setup3}. There exists a minimal generating set $(\phi_1 , \dots , \phi_{2s+1})$ for $I_2$ such that
$$I = (\phi_1 , \dots , \phi_{2s} ) + R_+ \phi_{2s+1}$$
\end{prop}

\begin{proof}
By the definition of a compressed ring,
\begingroup\allowdisplaybreaks
\begin{align*}
    \dim I_s &=\dim R_s - \min \{ \dim_k R_s , \dim_k R_0 + \dim_k R_{s-2} \} \\
    &= (s+2)(s+1)/2 - 1 - s(s-1)/2 \\
    &= 2s \\
\end{align*}
\endgroup
Choose a basis $\{ \phi_1 , \dots , \phi_{2s} \}$ for $(I_1 \cap I_2)_s$. By Proposition \ref{alscomp2}, $I_2$ defines a compressed ring, whence $\dim_k (I_2)_s = 2s+1$ by Lemma \ref{evenbtab}. Extend $\{\phi_1 , \dots , \phi_{2s} \}$ to a basis $\{ \phi_1 , \dots , \phi_{2s+1} \}$ of $(I_2)_s$. 

Observe that $(I_1 \cap I_2)_{s+1} = (I_2)_{s+1}$ by a dimension count. Since $I$ defines a compressed ring, the degrees of its minimal generators are concentrated in $2$ consecutive degrees. This means that
$$I = (\phi_1 , \dots , \phi_{2s}) + R_+ \phi_{2s+1}.$$
\end{proof}

\begin{prop}\label{numgens2}
Adopt Setup \ref{setup3}. Then $\mu (I) = 2s$ or $2s+1$. 
\end{prop}

\begin{proof}
In view of Corollary \ref{numgenss} and Proposition \ref{genset}, it suffices to compute the rank of the map $q \otimes k$ as in Theorem \ref{12.6}. A priori, $\rank_k (q \otimes k) \leq 3$; assume first that $\rank_k ( q \otimes k ) = 0$. 

Assume $I_2$ arises from the pfaffians of some skew symmetric matrix $M$. Observe that $M$ has linear entries by Lemma \ref{evenbtab}. Counting degrees, one notices that $q$ must have degree $0$ entries. Therefore $q = 0$ identically if $q \otimes k =0$, implying that $M$ must have an entire row of zeroes. This is impossible, so $\rank_k (q \otimes k) \geq 1$. 

Assume instead that $\rank_k (q \otimes k ) = 1$. Without loss of generality, we may assume that $q : V^* \to Re_x \oplus Re_y \oplus Re_z$ is the map sending $e_{2s+1}^* \mapsto e_x$ and all other basis vectors to $0$. This means that $M$ has a row consisting of a single nonzero linear entry, $x$. But the resolution of $I_2$ implies that there is a relation of the form $x \cdot \phi = 0$, contradicting the fact that $R$ is a domain. Thus $\rank_k (q \otimes k) \geq 2$. 
\end{proof}

\begin{cor}\label{torboundev}
Adopt Setup \ref{setup3}. Then $I$ defines a compressed ring of Tor algebra class $G(r)$ for some $2s-3 \leq r \leq 2s-1$. 
\end{cor}

\begin{proof}
By Theorem \ref{toralg} combined with Proposition \ref{numgens2}, $I$ has Tor algebra class $G(r)$ for $r \geq 2s-3$ or $2s-2$. Observe that $\dim_k \tor_1(R/I , k)_{s+1} = 0$ or $1$ if $\mu(I) = 2s$ or $2s+1$, respectively. Counting ranks on the homogeneous strands of the long exact sequence of Tor associated to the short exact sequence
$$0 \to I_2 / I \to R/I \to R/I_2 \to 0$$
we deduce that $\dim_k \tor_2^R (R/I , k)_{s+1} = 2s-2$ or $2s-1$ if $\mu(I) = 2s$ or $2s+1$, respectively. In a similar manner to Proposition \ref{torach}, we examine the induced map
$$\delta : T_2 \to \hom (T_1 , T_3),$$
where $T_i := \tor_i^R (R/I , R )$. Observe that
$$(T_1) (T_2)_{s+2} \subset (T_3)_{\geq 2s+2} = 0$$
and
$$(T_1)_{s+1} (T_2)_{s+1} \subset (T_2)_{2s+2} = 0$$
whence the only nontrivial products can occur between $(T_1)_{s}$ and $(T_2)_{s+1}$. This means
$$\rank_k \delta \leq 2s-2 \ \textrm{or} \ 2s-1.$$
\end{proof}

\begin{question}
Let $R= k[x,y,z]$. Does there exist a homogeneous ideal defining a compressed ring with $\soc (R/I) = k(-s) \oplus k(-2s+2)$ such that either
\begin{enumerate}[(a)]
    \item $\mu(I) = 2s$ and $R/I$ has Tor algebra class $G(2s-2)$, or
    \item $\mu(I) = 2s+1$ and $R/I$ has Tor algebra class $G(2s-1)$?
\end{enumerate}
\end{question}

As Corollary \ref{torboundev} suggests, the numerology alone does not forbid ideals of the above form to exist.

\section*{Acknowledgements}

Thanks to Andy Kustin for helpful comments on various drafts of this paper.

\end{document}